\documentclass[reqno]{amsart}
\usepackage{amsmath,amssymb,amsthm,amsrefs, amscd}
\usepackage{xcolor}
\usepackage[english]{babel}
\usepackage[utf8]{inputenc}	

\newtheorem{thm}{Theorem}[section]
\newtheorem{pro}[thm]{Proposition}
\newtheorem{lem}[thm]{Lemma}
\newtheorem{cor}[thm]{Corollary}
\theoremstyle{definition}
\newtheorem{defi}[thm]{Definition}
\newtheorem{rem}[thm]{Remark}

\newtheorem{exa}[thm]{Example}

\newcommand{\vu}{\vspace{.1cm}}
\newcommand{\vd}{\vspace{.2cm}}

\newcommand{\due}{\underline{\overline{\#}}}

\allowdisplaybreaks

\title[]{Hopf algebras arising from partial (co)actions}

\author[Azevedo]{Danielle Azevedo}
\address[Azevedo]{Instituto Federal de Educação, Ciência e Tecnologia do Rio Grande do Sul, Brazil}
\email{danielle.azevedo@alvorada.ifrs.edu.br}

\author[Martini]{Grasiela Martini}
\address[Martini]{Universidade Federal do Rio Grande, Brazil}
\email{grasiela.martini@furg.br}

\author[Paques]{Antonio Paques}
\address[Paques]{Universidade Federal do Rio Grande do Sul, Brazil}
\email{paques@mat.ufrgs.br}

\author[Silva]{Leonardo Silva}
\address[Silva]{Universidade Federal do Rio Grande do Sul, Brazil}
\email{leonardoufpel@gmail.com}

\begin{document}

\allowdisplaybreaks
\begin{abstract}
	In this paper, extending the idea presented by M. Takeuchi in \cite{Takeuchi}, we introduce the notion of  partial matched pair $(H,L)$ involving the concepts of partial action and partial coaction between two Hopf algebras $H$ and $L$. Furthermore, we present necessary conditions for the corresponding bismash product $L\# H$ to generate a new Hopf algebra and, as illustration, a family of examples is provided.
\end{abstract}



\thanks{The fourth author was partially supported by CNPq, Brazil}

\maketitle

\section{Introduction}

Partial group action is a concept that appeared for the first time within the theory of operator algebras, more precisely in the study by R. Exel \cite{Exel} about $C^\ast$-algebras generated by partial isometries over Hilbert spaces. In \cite{Dokuchaev} R. Exel and M. Dokuchaev presented a purely algebraic version of this concept and in \cite{Ferrero} M. Dokuchaev, M. Ferrero and the third author developed a Galois theory in this latter setting \cite{Ferrero}, which motivated the arising of a partial Hopf theory developed by S. Caenepeel and K. Janssen in \cite{Caenepeel} while M. Alves and E. Batista \cite{Marcelo} extended the results from \cite{Dokuchaev} concerning mainly to globalizations  to the Hopf context. Those works led to other investigations in several new directions and settings (see, for instance, the recent and nice survey by M. Dokuchaev and within its extensive list of references \cite{Dokucha}).

\vu

The notion of a matched pair  $(H,L)$ of Hopf algebras was introduced in the literature by W. Singer \cite{Singer} in the graded case and, under the condition that $L$ was commutative and $H$ cocommutative, a new Hopf algebra was obtained. In \cite{Takeuchi}, M. Takeuchi considered the ungraded case and constructed a Hopf algebra called the bismash product of $L$ with $H$ and denoted by $L \# H$.

\vu

In \cite{Majid} S. Majid also presented results that generated new structures of Hopf algebras. He described a huge class of noncommutative and noncocommutative Hopf algebras, called bicrossed product of Hopf algebras. Like M. Takeuchi, S. Majid also defined matched pair $(H, L)$ and constructed a Hopf algebra from it, assuming a certain condition of compatibility between $L$ and $H$. Such a condition replaced the commutativity and cocommutativity conditions required by M. Takeuchi and generalized the matched pair defined by him.

\vu
In this paper we extend the matched pair due to M. Takeuchi to the context of partial actions and partial coactions (briefly, partial (co)actions) and present conditions to construct Hopf algebras from such a pair. Furthermore, we  provide a family of examples to illustrate our method, as well as,  highlight some expected classical properties of the structure of the Hopf algebras obtained in this way, concerning to semisimplicity and dualization.

\vu
The present work is divided as follows. In section 2 the concepts of partial (co)actions and matched pair will be briefly recalled.

\vu

In section 3 we introduce the definition of partial matched pair $(H, L)$ of Hopf algebras involving the structures of partial module algebra and partial comodule coalgebra and analyze under what conditions it is possible to construct a Hopf algebra involving the usual partial smash product, which we will denote by $L\due H$. Essentially,  we replace the commutativity of $L$ and the cocommutativity of $H$  proposed by M. Takeuchi by a weaker condition, namely  quasi-abelian (see Definition \ref{defabeliano}), and dedicate the rest of this section to construct necessary conditions to insure the existence of a family of partial (co)actions that generate Hopf algebras through quasi-abelian partial matched pairs (see, in particular, Theorem \ref{anti}).

\vu

In section 4 we present several examples of partial matched pair with the corresponding description of  the Hopf algebras obtained from them.  We end this section making still some considerations about the dualization of the structure of such algebras, as well as, the existence of nonzero integrals within these ones.

\vu

As a final remark, we observe that, with a similar approach, it is also possible to extend the notion of  matched pair due to S. Majid \cite{Majid} to the context of partial (co)actions, as well as, to generate a new Hopf algebra from its bicrossed product.

\vu

Throughout, vector spaces and (co)algebras will be all considered over a fixed field $\Bbbk$. Unadorned $\otimes$ means $\otimes_{\Bbbk}$. Moreover, in this text (co)algebras are not necessarily (co)unital.

\section{Preliminaries}

\subsection{Partial (co)actions} In this section we will recall the notions of partial actions and coactions of Hopf algebras.

\begin{defi} \label{part_mod_alg}\cite[Proposition 4.5]{Caenepeel} Let $H$ be a Hopf algebra. A unital algebra $A$ is a \textit{left partial $H$-module algebra} if there is a linear map $\rightharpoonup: H \otimes A  \longrightarrow  A$, denoted by $\rightharpoonup(h \otimes a)= h\rightharpoonup a$, such that
	\begin{itemize}
	\item [(i)] $1_H\rightharpoonup a=a$;
	\item [(ii)] $h\rightharpoonup ab=(h_1\rightharpoonup a)(h_2\rightharpoonup b)$;
	\item [(iii)] $h\rightharpoonup(g\rightharpoonup a)=(h_1\rightharpoonup 1_A)(h_2g\rightharpoonup a)$,
	\end{itemize}
for all $h,g\in H$ and $a,b\in A$. In this case $\rightharpoonup$ is called a  \emph{left partial action} of $H$ on $A$. Moreover, we say that $\rightharpoonup$  is \textit{symmetric} if the following additional condition holds
\begin{itemize}
\item[(iv)] $h\rightharpoonup (g\rightharpoonup a)=(h_1g\rightharpoonup a)(h_2\rightharpoonup 1_A)$.
\end{itemize}
Right partial $H$-module algebras are defined similarly.
\end{defi}
Every left (global) $H$-module algebra is a partial one. In particular, a left partial $H$-module algebra is global if and only if  $h\rightharpoonup 1_A = \varepsilon_H(h)1_A$, for all $h\in H$.

\begin{exa}\label{mod_acao_corpo} \cite[Lemma 4.1]{Felipe}
Let $H$ be a Hopf algebra, $A$ a unital algebra and $\lambda:H \longrightarrow \Bbbk$ a linear map. Then the map $\rightharpoonup:H\otimes A\to A$ given by $h\rightharpoonup a= \lambda(h)a$, for all $h\in H$ and $a\in A$, is a left partial action of $H$ on $A$ if and only if $\lambda(1_H)=1_{\Bbbk}$ and $\lambda(h)\lambda(g)=\lambda(h_1)\lambda(h_2g)$, for all $g,h\in H$.
\end{exa}
			
Given any left partial $H$-module algebra $A$, we can endow the tensor product $A \otimes H$ with an algebra structure induced by
$$(a\otimes h)(b\otimes g)=a(h_1\rightharpoonup b)\otimes h_2g,$$
for $a,b\in A$ e $h,g\in H$. We denote this structure by $A \# H$ and we call it the smash product of $A$ with $H$. In this case, $A\# H$ is just an algebra with left unit $1_A\# 1_H$.
			
	\begin{lem}\label{algebra_smash} \cite[Section 2]{Muniz}
	The vector subspace $A\underline{\#}H=(A\# H)(1_A\# 1_H)$ is a unital algebra with the multiplication induced by
	$$(x\underline{\#}h)(y\underline{\#}g)=x(h_1\rightharpoonup y)\underline{\#} h_2g,$$
	for all $x,y\in A$ and $h,g\in H$, and unit given by $1_A\underline{\#}1_H$. The unital algebra $A\underline{\#} H$ is called the partial smash product of $A$ with $H$.
	\end{lem}
			
	Next we have the definition of partial coaction of a Hopf algebra $H$ on a counital coalgebra $C$.
			
	\begin{defi}\cite[Definition 6.1]{Batista}\label{part_comod_coalg} Let $H$ be a Hopf algebra. A counital coalgebra $C$ is  a \textit{right partial $H$-comodule coalgebra} if there exists a linear map $\rho: C \longrightarrow  C \otimes H$ denoted by $\rho(c)=c^0 \otimes c^1$, for all $c\in C$, satisfying
	\begin{itemize}
	\item [(i)] $(id_C\otimes \varepsilon_H)\rho(c)=c$;
	\item [(ii)] $(\Delta_C\otimes id_H)\rho(c)={c_1}^0\otimes {c_2}^0\otimes {c_1}^1 {c_2}^1$;
	\item [(iii)] $(\rho \otimes id_H)\rho(c)={c_1}^0\otimes {{c_1}^1}_1\otimes {{c_1}^1}_2\varepsilon_C({c_2}^0){c_2}^1$.
	\end{itemize}
	In this case $\rho$ is called a \emph{right partial coaction} of $H$ on $C$. If, in addition
\begin{itemize}
\item[(iv)] $(\rho \otimes id_H)\rho(c)= {c_2}^0\otimes {{c_2}^1}_1\otimes {{c_2}^1}_2\varepsilon_C({c_1}^0){c_1}^1$
 \end{itemize} holds, then we say that the $\rho$ is \textit{symmetric}. In a similar way one defines a left partial $H$-comodule coalgebra.
\end{defi}
						
	Every right (global) $H$-comodule coalgebra is a partial one. Moreover, a right partial $H$-comodule coalgebra $C$ is global if and only if $\varepsilon_C(c^0)c^1 = \varepsilon_C(c)1_H$, for all $c\in C$.
						
	\begin{exa} \label{prop cc}\cite[Proposition 4.12]{Glauber}
	Let $H$ be a Hopf algebra, $z\in H$ and $C$ a counital coalgebra. Then the linear map $\rho: C \longrightarrow C \otimes H$ given by $\rho(c)= c \otimes z$ is a right partial coaction if and only if $\varepsilon_H(z)= 1_{\Bbbk}$ and $z \otimes z = \Delta_H(z)(1 \otimes z)$.
	Note that if the element $z$ satisfies the earlier conditions, then $z^2 = z$.
	\end{exa}
							
\begin{rem} \label{coalgebra cosmash} Let $H$ be a Hopf algebra and $C$ a right partial $H$-comodule coalgebra. Then the vector space $H\otimes C$ is a right counital coalgebra given by
	$$\Delta(h\otimes c)=(h_1\otimes {c_1}^0)\otimes (h_2{c_1}^1\otimes c_2) \ \mbox{and} \  \varepsilon(h\otimes c)=\varepsilon_H(h)\varepsilon_C(c),$$ for all $c\in C$ and $h\in H$.  We denoted this structure by $H\rtimes C$.
\end{rem}

Note that $(\varepsilon\otimes id_{H\otimes C})\Delta(h\otimes c)=h\varepsilon_C({c_1}^0){c_1}^1\otimes c_2$. Thus, when $C$ is a right $H$-comodule coalgebra the map $\varepsilon$ is a counit. In \cite{Batista} the authors constructed the partial smash coproduct associated to a right partial $H$-comodule coalgebra in the following way.

\begin{pro} \label{coalgebra cosmash2} Let $C$ be a right partial $H$-comodule coalgebra. Then the subspace
$H\overline{\rtimes }C=(\varepsilon\otimes id_{H\otimes C})\Delta(H\rtimes C)$ is a counital coalgebra induced by
	$$\Delta(h\overline{\rtimes } c)=(h_1\overline{\rtimes } {c_1}^0)\otimes (h_2{c_1}^1\overline{\rtimes } c_2) \ \mbox{and} \ \varepsilon(h\overline{\rtimes } c)=\varepsilon_H(h)\varepsilon_C(c),$$
for all $c\in C$ and $h\in H$. Each element of $H\overline{\rtimes }C$ is written as follows $h\overline{\rtimes }c=\varepsilon_C({c_1}^0)h{c_1}^1\otimes c_2.$
\end{pro}

\subsection{Matched pair}

In \cite{Takeuchi}, M. Takeuchi constructed Hopf algebras using (global) (co)actions under certain conditions. Next we recall those results.

\begin{defi}\label{matched pair}\cite[Definition 1.1]{Takeuchi} A pair of Hopf algebras $(H,L)$ is called a matched pair if
	\begin{itemize}
	\item[(i)] $L$ is a left $H$-module algebra with action given by $h\triangleright x$;
	\item[(ii)] $H$ is a right $L$-comodule coalgebra with structure $\rho:H\longrightarrow H\otimes L$, such that $\rho(h)=h^0\otimes h^1$;
	\item[(iii)] $\rho(hg)=\rho(h_1)(g^0\otimes (h_2\triangleright g^1))$;
	\item[(iv)] $\Delta(h\triangleright x)=({h_1}^0\triangleright x_1)\otimes ({h_1}^1(h_2\triangleright x_2)) $,
	\end{itemize}
	for all $h,g\in H$ and $x\in L$. The matched pair $(H,L)$ is called \emph{abelian} if $H$ is cocommutative
	and $L$ is commutative.
	\end{defi}

\begin{lem} \label{global_pair}\cite[Lemma 1.2]{Takeuchi}
	Let $(H,L)$ be a matched pair. Then, $$\rho(1_H)=1_H\otimes 1_L \ \ \ \mbox{and} \ \ \ \varepsilon_L(h\triangleright x)=\varepsilon_H(h)\varepsilon_L(x),$$
	for all $h\in H$ and $x\in L$.
\end{lem}

\begin{thm}\cite[Theorem 1.4]{Takeuchi}
	Let $(H,L)$ be an abelian matched pair. Then $L\# H$ is a Hopf algebra with the structure given by the following:
	\begin{itemize}
		\item[(i)] Multiplication: $(x\# h)(y \# g)= x(h_1\triangleright y)\# h_2 g$;
		\item[(ii)] Unit: $1_L \# 1_H$;
		\item[(iii)] Comultiplication: $\Delta(x\# h)=x_1 \# {h_1}^0 \otimes x_2 {h_1}^1 \# h_2$;
		\item[(iv)] Counit: $\varepsilon(x\# h)= \varepsilon_L(x)\varepsilon_H(h)$;
		\item [(v)] Antipode: $S(x\# h)=(1_L\# S_H(h^0))(S_L(h^1)S_L(x)\#1_H)$.
	\end{itemize}
	for all $x,y\in L$ and $h,g\in H$.
\end{thm}

The Hopf algebra $L\# H$ is called the \textit{bismash product of $L$ with $H$}.

\section{Partial Matched Pair}

In this section we extend the concept of matched pair to the context of partial (co)actions.

\begin{defi}\label{partial matched pair} A pair of Hopf algebras $(H,L)$ is called a \textit{partial matched pair} if
	\begin{enumerate}
		\item[(i)] $L$ is a left partial $H$-module algebra via $\rightharpoonup$;
		\item[(ii)] $H$ is a right partial $L$-comodule coalgebra via $\rho$, which is denoted by $\rho(h)= h^0 \otimes h^1$;
		\item[(iii)] $(\tau_{L,H}\otimes m_L)(id_L\otimes \tau_{L,H}\otimes id_L)(\Delta_L(h_1\rightharpoonup x)\otimes \rho(h_2g))$\\ \vspace{-0.4cm}
		\begin{flushright}
		=\ ${h_3}^0 g^0\otimes ({h_1}^0\rightharpoonup x_1)\otimes {h_1}^1(h_2\rightharpoonup x_2){h_3}^1(h_4\rightharpoonup g^1)$;
	\end{flushright}
		\item[(iv)] $\varepsilon_L(h\rightharpoonup x)=\varepsilon_L(h\rightharpoonup 1_L)\varepsilon_L(x)$;
		\item [(v)] $\rho(1_H)=1_H\otimes \varepsilon_H({1_H}^0){1_H}^1$,	
	\end{enumerate}
for all $h,g\in H$ and $x\in L$, where $\tau_{L,H}: L\otimes H\longrightarrow H\otimes L$ is the twist map.
\end{defi}

Similarly, one defines a partial matched pair $(L,H)$ considering $L$ a right partial $H$-module algebra and $H$ a left partial $L$-comodule coalgebra.

\vu

In all what follows $(H,L)$ will be considered as in  Definition \ref{partial matched pair}, unless otherwise stated.

\begin{rem}\label{troca}
	The item (iii) of the above definition can be rewritten as follows:
	\begin{eqnarray*}
	(h_2 g)^0\otimes (h_1\rightharpoonup x)_1\otimes (h_1\rightharpoonup x)_2 (h_2 g)^1&=&{h_3}^0 g^0
	\otimes ({h_1}^0\rightharpoonup x_1)\otimes \\
	&& {h_1}^1(h_2\rightharpoonup x_2){h_3}^1(h_4\rightharpoonup g^1),
	\end{eqnarray*}
	for all $h,g\in H$ and $x\in L$.
\end{rem}

\begin{pro}
	Let $(H,L)$ be a matched pair. Then $(H,L)$ is a partial matched pair.
\end{pro}
\begin{proof}
	It is follows directly from Definition \ref{matched pair} and Lemma \ref{global_pair}.
\end{proof}

The result below tells us when a partial matched pair is global.

\begin{thm}
Let $(H,L)$ be a partial matched pair. Then $(H,L)$ is a matched pair if and only if
	\begin{enumerate}
		\item[(i)] $h\rightharpoonup 1_L= \varepsilon_H(h)1_L$;
		\item[(ii)] $\varepsilon_H(h^0) h^1=\varepsilon_H(h)1_L$,
	\end{enumerate}
 hold for all $h\in H$.
\end{thm}

\begin{proof}
	If $(H,L)$ is a matched pair, it is clear that we have the items (i) and (ii). Conversely, supposing that the items (i) and (ii) are satisfied, we need to show that,
	$$\rho(hg)= {h_1}^0 g^0 \otimes {h_1}^1 (h_2\rightharpoonup g^1) \ 	\mbox{and} \ \Delta (h\rightharpoonup x)= ({h_1}^0\rightharpoonup x_1) \otimes {h_1}^1 (h_2\rightharpoonup x_2),$$
	for all $h,g\in H$ and $x\in L$. Indeed,
	\begin{eqnarray*}
		(hg)^0 \otimes 1_L \otimes (hg)^1 &=& (h_2g)^0 \otimes \varepsilon_H(h_1)1_L \otimes 1_L (h_2 g)^1\\
		&=& (h_2 g)^0 \otimes (h_1\rightharpoonup 1_L)_1 \otimes (h_1\rightharpoonup 1_L)_2 (h_2 g)^1\\
		&\stackrel{\ref{troca}}{=}& {h_3}^0 g^0 \otimes ({h_1}^0\rightharpoonup 1_L) \otimes {h_1}^1 (h_2\rightharpoonup 1_L){h_3}^1 (h_4\rightharpoonup g^1)\\
		&=& {h_1}^0 g^0 \otimes 1_L \otimes {h_1}^1 (h_2 \rightharpoonup g^1),
	\end{eqnarray*}
	thus $\rho(hg)= {h_1}^0 g^0 \otimes {h_1}^1 (h_2\rightharpoonup g^1).$ And,
	\begin{eqnarray*}
		\Delta(h\rightharpoonup x) &=& (h\rightharpoonup x)_1 \otimes (h\rightharpoonup x)_2\\
		&=& \varepsilon_H({h_2}^0)(h_1\rightharpoonup x)_1 \otimes (h_1 \rightharpoonup x)_2 {h_2}^1\\
		&\stackrel{\ref{troca}}{=}& \varepsilon_H({h_3}^0 {1_H}^0)({h_1}^0 \rightharpoonup x_1)\otimes {h_1}^1 (h_2\rightharpoonup x_2) {h_3}^1 (h_4\rightharpoonup {1_H}^1)\\
		&=& ({h_1}^0 \rightharpoonup x_1)\otimes {h_1}^1 (h_2\rightharpoonup x_2).
	\end{eqnarray*}\end{proof}

\begin{defi} \label{defabeliano} We say that $(H,L)$ is a \textit{quasi-abelian} partial matched pair if
	\begin{equation*}
	{h_2}^0\otimes (h_1\rightharpoonup x) {h_2}^1 = {h_1}^0\otimes {h_1}^1(h_2\rightharpoonup x),
	\end{equation*}
	for all $h\in H$ and $x\in L$. Clearly, abelian matched pair is quasi-abelian.
\end{defi}

\begin{lem}\label{exemplo} Let $L$ be the left partial $H$-module algebra given by Example \ref{mod_acao_corpo} and let $H$ be the right partial $L$-comodule coalgebra of Example \ref{prop cc}.	 Then $(H,L)$ is a quasi-abelian partial matched pair if and only if $z$ is a central element in $L$ and $\lambda(h_1)h_2 = h_1\lambda(h_2)$, for all $h\in H$.
\end{lem}
\begin{proof} Suppose that $(H,L)$ is a quasi-abelian partial matched pair. It is follows directly from Definition \ref{defabeliano} that $z$ is a central element in $L$ and $\lambda(h_1)h_2 = h_1\lambda(h_2)$, for all $h\in H$. Conversely, to prove that $(H,L)$ is a quasi-abelian partial matched pair, it is enough to show that the condition (iii) of Definition \ref{partial matched pair}, the others are immediate. Let $h,g\in H$ and $x\in L$,
	\begin{eqnarray*}
		(h_2g)^0 \otimes (h_1 \rightharpoonup x)_1 \otimes (h_1 \rightharpoonup x)_2 (h_2g)^1 &=& \lambda(h_1)h_2g \otimes x_1 \otimes x_2z\\
		&=& \lambda(h_1)\lambda(h_2)\lambda(h_3)h_4g \otimes x_1 \otimes x_2z^3\\
		&=& \lambda(h_1)\lambda(h_2)h_3g\lambda(h_4)\otimes x_1 \otimes zx_2z^2\\
		&=& {h_3}^0g^0 \otimes ({h_1}^0 \rightharpoonup x_1) \\
		&& \otimes \ {h_1}^1(h_2 \rightharpoonup x_2) {h_3}^1(h_4 \rightharpoonup g^1).
	\end{eqnarray*}
	
	Furthermore, we have
	\begin{eqnarray*}
		{h_2}^0\otimes (h_1\rightharpoonup x) {h_2}^1 &=& h_2\otimes \lambda(h_1)xz\\
		&=& h_1 \otimes \lambda(h_2)zx\\
		&=& {h_1}^0\otimes {h_1}^1(h_2\rightharpoonup x).
	\end{eqnarray*}
	Thus, the proof is complete.
\end{proof}

Before to proceed to the next proposition, recall the map defined in Remark \ref{coalgebra cosmash}, namely $\Delta: H \otimes C \longrightarrow (H \otimes C) \otimes (H \otimes C)$ given by  $\Delta(h\otimes c)=(h_1\otimes {c_1}^0)\otimes (h_2{c_1}^1\otimes c_2)$.
In particular, $H \otimes C = H \# C$ as vector spaces, then it makes sense to consider $\Delta: H \# C \longrightarrow (H \# C) \otimes (H \# C)$ given by $\Delta(h \# c)=\Delta(h \otimes c)$.
In the same way, we may consider $\varepsilon$ defined in $H \# L$ as the same as defined in $ H \otimes L$.
Until now, the Lemma \ref{algebra_smash} and Proposition \ref{coalgebra cosmash2} are used to build structures of unital algebra and counital coalgebra in certain subspaces of $A \otimes H$ and $H \otimes C$, respectively.
What we want is to build a Hopf algebra structure in a certain subspace of $H \otimes L$ using precisely such two previous constructions, and so
the structure maps we consider are the same as those used therein.

 \begin{pro}\label{delta_multi}
	Let $(H,L)$ be a quasi-abelian partial matched pair and consider the algebra $L\# H$. Then the linear map $\Delta$ defined in Remark \ref{coalgebra cosmash} is multiplicative.
\end{pro}
\begin{proof}
	Indeed, for all $x,y\in L$ and $h,g\in H$,
	\begin{eqnarray*}
		\Delta((x\# h)(y\# g))\hspace{-0.2cm}&=& \Delta(x(h_1\rightharpoonup y)\# h_2g)\\
		&=& x_1(h_1\rightharpoonup y)_1\# (h_2g_1)^0\otimes x_2(h_1\rightharpoonup y)_2(h_2g_1)^1\# h_3g_2\\
		&\stackrel{\ref{troca}}{=}& x_1({h_1}^0\rightharpoonup y_1)\# {h_3}^0 {g_1}^0 \otimes x_2 {h_1}^1(h_2\rightharpoonup y_2){h_3}^1(h_4\rightharpoonup {g_1}^1)\# h_5 g_2\\
		&\stackrel{3.5}{=}& x_1({h_1}^0\rightharpoonup y_1)\# {h_2}^0 {g_1}^0 \otimes x_2 {h_1}^1{h_2}^1(h_3\rightharpoonup y_2)(h_4\rightharpoonup {g_1}^1)\# h_5 g_2\\
		&=& x_1({{h_1}^0}_1\rightharpoonup y_1)\# {{h_1}^0}_2{g_1}^0\otimes x_2{h_1}^1(h_2\rightharpoonup y_2{g_1}^1)\# h_3g_2\\
		&=& (x_1\# {h_1}^0)(y_1\# {g_1}^0)\otimes (x_2{h_1}^1\# h_2)(y_2 {g_1}^1\# g_2)\\
		&=& \Delta(x\# h)\Delta(y\# g).
	\end{eqnarray*}
\end{proof}

\begin{lem}\label{igualdade importante} Let $(H,L)$ be a quasi-abelian partial matched pair. Then
	\begin{enumerate}
		\item [(i)] $(h\rightharpoonup 1_L)=(h_1\rightharpoonup 1_L)\varepsilon_L(h_2 \rightharpoonup 1_L)= \varepsilon_L(h_1\rightharpoonup 1_L)(h_2\rightharpoonup 1_L)$;
		\item [(ii)]  $x_1(h_1 \rightharpoonup 1_L)_1\otimes {h_2}^0\otimes x_2(h_1 \rightharpoonup 1_L)_2{h_2}^1\otimes h_3=x_1({h_1}^0 \rightharpoonup 1_L)\otimes {h_2}^0 	\\
	 	\otimes x_2{h_1}^1{h_2}^1(h_3 \rightharpoonup 1_L)\otimes h_4$,
	\end{enumerate}
for all $x\in L$ and $h\in H$.
	\end{lem}
\begin{proof} (i) Let $h\in H$,
	\begin{eqnarray*}
		h \rightharpoonup 1_L\hspace{-0.4cm} &=& \hspace{-0.3cm}\varepsilon_H({h_2}^0) (h_1\rightharpoonup 1_L)_1 \varepsilon_L((h_1\rightharpoonup 1_L)_2) \varepsilon_L({h_2}^1)\\
		&=&\hspace{-0.3cm} (\varepsilon_H \otimes id_L\otimes \varepsilon_L) ({h_2}^0 \otimes (h_1\rightharpoonup 1_L)_1 \otimes (h_1\rightharpoonup 1_L)_2 {h_2}^1)\\
		&\stackrel{\ref{troca}}{=}& \hspace{-0.3cm}(\varepsilon_H \otimes id_L \otimes \varepsilon_L)({h_3}^0 {1_H}^0\otimes ({h_1}^0 \rightharpoonup 1_L)\otimes {h_1}^1(h_2 \rightharpoonup 1_L){h_3}^1 (h_4 \rightharpoonup {1_H}^1))\\
		&\stackrel{\ref{partial matched pair}}{=}& \hspace{-0.3cm} \varepsilon_H({h_3}^0 {1_H}^0) ({h_1}^0 \rightharpoonup \hspace{-0.05cm}1_L) \varepsilon_L({h_1}^1) \varepsilon_L(h_2 \rightharpoonup \hspace{-0.05cm}1_L) \varepsilon_L({h_3}^1) \varepsilon_L(h_4 \rightharpoonup \hspace{-0.05cm}1_L)\varepsilon_L({1_H}^1)\\
		&=& \hspace{-0.3cm}\varepsilon_H(h_3 1_H) (h_1\rightharpoonup 1_L)\varepsilon_L(h_2\rightharpoonup 1_L)\varepsilon_L(h_4\rightharpoonup 1_L)\\
		&=& \hspace{-0.3cm}(h_1\rightharpoonup 1_L)\varepsilon_L(h_2\rightharpoonup 1_L).
	\end{eqnarray*}
	And,
	\begin{eqnarray*}
		h\rightharpoonup 1_L &=& (h_1\rightharpoonup 1_L)\varepsilon_L(h_2\rightharpoonup 1_L)=
	({h_1}^0\rightharpoonup 1_L)\varepsilon_L({h_1}^1(h_2\rightharpoonup 1_L))\\
		&\stackrel{3.5}{=}&({h_2}^0\rightharpoonup 1_L)\varepsilon_L((h_1\rightharpoonup 1_L){h_2}^1)=
	(h_2\rightharpoonup 1_L)\varepsilon_L(h_1\rightharpoonup 1_L).
	\end{eqnarray*}	

\vd
	
	(ii) Let $h\in H$ and $x\in L$,
	\begin{eqnarray*}
		&\hspace{-0.4cm}&\hspace{-0.65cm} x_1(h_1 \rightharpoonup 1_L)_1\otimes {h_2}^0\otimes x_2(h_1 \rightharpoonup 1_L)_2{h_2}^1\otimes h_3\\
		&\hspace{-0.4cm}\stackrel{\ref{troca}}{=}& \hspace{-0.3cm}x_1({h_1}^0 \rightharpoonup 1_L)\otimes {h_3}^0 {1_H}^0\otimes x_2{h_1}^1(h_2 \rightharpoonup 1_L){h_3}^1(h_4 \rightharpoonup {1_H}^1) \otimes h_5\\
		&\hspace{-0.4cm}\stackrel{\ref{igualdade importante}}{=}& \hspace{-0.3cm}x_1({{h_1}^0}_1 \rightharpoonup 1_L)\varepsilon_L({{h_1}^0}_2 \rightharpoonup 1_L)\otimes {h_3}^0 {1_H}^0 \otimes x_2 {h_1}^1(h_2 \rightharpoonup 1_L){h_3}^1 (h_4 \rightharpoonup {1_H}^1)\otimes h_5\\
		&\hspace{-0.4cm}=& \hspace{-0.3cm}x_1({h_1}^0 \rightharpoonup\hspace{-0.03cm} 1_L) \varepsilon_L({h_2}^0 \rightharpoonup \hspace{-0.03cm}1_L) \otimes {h_4}^0 {1_H}^0\otimes x_2 {h_1}^1 {h_2}^1 (h_3 \rightharpoonup \hspace{-0.03cm}1_L){h_4}^1 (h_5 \rightharpoonup \hspace{-0.03cm}{1_H}^1) \otimes h_6\\
		&\hspace{-0.4cm}\stackrel{\ref{troca}}{=}&  \hspace{-0.3cm}x_1({h_1}^0 \rightharpoonup 1_L) \varepsilon_L((h_2\rightharpoonup 1_L)_1)\otimes {h_3}^0 \otimes x_2 {h_1}^1 (h_2\rightharpoonup 1_L)_2 {h_3}^1\otimes h_4\\	
		&\hspace{-0.4cm}=&\hspace{-0.3cm} x_1 ({h_1}^0\rightharpoonup 1_L) \otimes {h_3}^0 \otimes x_2{h_1}^1 (h_2\rightharpoonup 1_L){h_3}^1\otimes h_4\\
		&\hspace{-0.4cm}\stackrel{3.5}{=} & \hspace{-0.3cm} x_1({h_1}^0 \rightharpoonup 1_L)\otimes {h_2}^0 \otimes x_2{h_1}^1{h_2}^1(h_3 \rightharpoonup 1_L)\otimes h_4.
	\end{eqnarray*}
	\end{proof}
As a consequence of this previous lemma we have the following result.

\begin{pro}\label{barra embaixo coalgebra}  Let $(H,L)$ be a quasi-abelian partial matched pair. Then $L\underline{\#} H$ is a unital algebra and a right counital coalgebra with the following structure maps
$$\Delta(x\underline{\#}h)=x_1\underline{\#}{h_1}^0\otimes x_2{h_1}^1\underline{\#}h_2  \ \ \mbox{and} \ \ \varepsilon(x\underline{\#}h)=\varepsilon_L(x)\varepsilon_L(h \rightharpoonup 1_L),$$
for all $h\in H$ and $x\in L$. Moreover, the linear map $\Delta$ is multiplicative and  $\varepsilon$ is  a morphism of algebras.
\end{pro}
\begin{proof} Note that by Lemma \ref{algebra_smash}, we have that $L\underline{\#} H$ is an algebra with unit $1_L\underline{\#}1_H$. Now, using Remark \ref{coalgebra cosmash},
\begin{eqnarray*}
\Delta(x\underline{\#}h) &=& \Delta(x(h_1 \rightharpoonup 1_L)\# h_2)\\
&=& x_1(h_1 \rightharpoonup 1_L)_1\# {h_2}^0\otimes x_2(h_1 \rightharpoonup 1_L)_2{h_2}^1\# h_3\\
&\stackrel{\ref{igualdade importante}}{=}& x_1({h_1}^0 \rightharpoonup 1_L)\# {h_2}^0 \otimes x_2{h_1}^1{h_2}^1(h_3 \rightharpoonup 1_L)\# h_4\\
&=& x_1({{h_1}^0}_1 \rightharpoonup 1_L)\# {{h_1}^0}_2\otimes x_2 {h_1}^1(h_2 \rightharpoonup 1_L)\# h_3\\
&=& x_1 \underline{\#} {h_1}^0 \otimes x_2 {h_1}^1\underline{\#} h_2 \in (L\underline{\#}H)\otimes (L\underline{\#}H),
\end{eqnarray*}
for all $h\in H$ and $x\in L$.
Furthermore,
\begin{align*}
\varepsilon(x\underline{\#}h) &= \varepsilon(x(h_1 \rightharpoonup 1_L)\# h_2) \\
& =
\varepsilon_L(x(h_1 \rightharpoonup 1_L)) \varepsilon_H(h_2) \\
& = \varepsilon_L(x) \varepsilon_L(h_1 \varepsilon_H(h_2) \rightharpoonup 1_L)\\
& =
\varepsilon_L(x) \varepsilon_L(h \rightharpoonup 1_L),
\end{align*} and so
\begin{eqnarray*}
( id_{L \underline{\#} H}  \otimes \varepsilon )(\Delta (x \underline{\#} h)) & = &( id_{L \underline{\#} H}  \otimes \varepsilon )( x_1 \underline{\#} {h_1}^0 \otimes x_2 {h_1}^1\underline{\#} h_2) \\
& =&  x_1 \underline{\#} {h_1}^0 \varepsilon(x_2 {h_1}^1\underline{\#} h_2) \\
&=& x_1\underline{\#} {h_1}^0 \varepsilon_L(x_2)\varepsilon_L({h_1}^1)\varepsilon_L(h_2 \rightharpoonup 1_L)\\
&=& x\underline{\#} {h_1}^0 \varepsilon_L({h_1}^1)\varepsilon_L(h_2\rightharpoonup 1_L)\\
&=& x\underline{\#} {h_1}^0 \varepsilon_L({h_1}^1(h_2\rightharpoonup 1_L))\\
&\stackrel{3.5}{=}& x\underline{\#} {h_2}^0 \varepsilon_L((h_1\rightharpoonup 1_L){h_2}^1)\\
&=& x\underline{\#}{h_2}^0 \varepsilon_L(h_1\rightharpoonup 1_L)\varepsilon_L({h_2}^1)\\
&\stackrel{\ref{part_comod_coalg}(i)}{=}& x\underline{\#} h_2\varepsilon_L(h_1\rightharpoonup 1_L)\\
&=& x(h_2\rightharpoonup 1_L)\# h_3\varepsilon_L(h_1\rightharpoonup 1_L)\\
&\stackrel{\ref{igualdade importante}(i)}{=}& x(h_1\rightharpoonup 1_L)\# h_2\\
&=& x\underline{\#} h.
\end{eqnarray*}

Therefore, $L\underline{\#} H$ is a right counital coalgebra.

\vu

By Proposition \ref{delta_multi} $\Delta$ is multiplicative in the whole space $ L \# H$, in particular also multiplicative in the subspace $L \underline{\#} H$. To conclude, it is enough to check that $\varepsilon$ is a morphism of algebras. Indeed, for $(x\underline{\#}h)$, $(y\underline{\#}g)\in L\underline{\#} H$,
\begin{eqnarray*}
\varepsilon((x\underline{\#}h)(y\underline{\#}g))&=& \varepsilon(x(h_1 \rightharpoonup y)\underline{\#} h_2g)\\
&=& \varepsilon_L(x)\varepsilon_L(h_1 \rightharpoonup y)\varepsilon_L(h_2g\rightharpoonup 1_L)\\
&\stackrel{\ref{partial matched pair}}{=}& \varepsilon_L(x)\varepsilon_L(h_1\rightharpoonup 1_L)\varepsilon_L(y)\varepsilon_L(h_2g\rightharpoonup 1_L)\\
&=& \varepsilon_L(x)\varepsilon_L(y)\varepsilon_L((h_1 \rightharpoonup 1_L)(h_2g \rightharpoonup 1_L))\\
&=& \varepsilon_L(x)\varepsilon_L(y)\varepsilon_L(h\rightharpoonup (g\rightharpoonup 1_L))\\
&\stackrel{\ref{partial matched pair}}{=}& \varepsilon_L(x)\varepsilon_L(y)\varepsilon_L(h \rightharpoonup 1_L)\varepsilon_L(g\rightharpoonup 1_L)\\
&=& \varepsilon(x\underline{\#} h)\varepsilon(y\underline{\#}g)
\end{eqnarray*}
and, $\varepsilon(1_L\underline{\#} 1_H)=\varepsilon(1_L(1_H\rightharpoonup 1_L)\# 1_H)=\varepsilon_L(1_L)\varepsilon_L(1_H\rightharpoonup 1_L)\varepsilon_H(1_H)= 1_{\Bbbk}. $
\end{proof}

 Note that $L\underline{\#} H$ is not necessarily a counital coalgebra for, in general,
\begin{eqnarray*}
(\varepsilon\otimes id_{L \underline{\#} H})\Delta(x\underline{\#}h)&=& (\varepsilon \otimes id_{L \underline{\#} H})(x_1 \underline{\#} {h_1}^0 \otimes x_2{h_1}^1\underline{\#} h_2)\\
&=& \varepsilon(x_1 \underline{\#} {h_1}^0) x_2{h_1}^1\underline{\#} h_2\\
&=& \varepsilon_L(x_1)\varepsilon_L({h_1}^0\rightharpoonup 1_L) x_2{h_1}^1 (h_2\rightharpoonup 1_L)\# h_3\\
&=&\varepsilon_L({h_1}^0\rightharpoonup 1_L) x{h_1}^1 (h_2\rightharpoonup 1_L)\# h_3\\
&\neq& x\underline{\#}h.
\end{eqnarray*}
The next example illustrates this fact.

\begin{exa}
Let $L$ be the left partial $H$-module algebra given in Example \ref{mod_acao_corpo} and let $H$ be the right partial $L$-comodule coalgebra of Example \ref{prop cc}, where $z$ is a central element in $L$ and $\lambda(h_1)h_2 = h_1\lambda(h_2)$, for all $h\in H$. By Lemma \ref{exemplo}, $(H,L)$ is a quasi-abelian partial matched pair. However, $$(\varepsilon\otimes id_{L \underline{\#} H})\Delta(x\underline{\#}h)= \lambda(h_1)xz\# h_2\quad\text{and}\quad x\underline{\#}h=\lambda(h_1)x\#h_2,$$ i.e, $(\varepsilon\otimes id_{L \underline{\#} H})\Delta(x\underline{\#}h)\neq x\underline{\#}h.$
\end{exa}
Therefore, we will consider the subspace
$$\overline{L\underline{\#}H}= (\varepsilon\otimes id_{L \underline{\#} H})\Delta(L\underline{\#}H),$$
 where each element is written as $$\overline{x\underline{\#}h}=\varepsilon_L({h_1}^0\rightharpoonup 1_L) x{h_1}^1\underline{\#}h_2,$$ with $h\in H$ and $x\in L$.

\vu

In order to verify that this subspace is a bialgebra we will need the next result.
\begin{lem}\label{importante}
Let $(H,L)$ be a quasi-abelian partial matched pair. Then
$$\varepsilon_H({h_1}^0){h_1}^1(h_2\rightharpoonup 1_L)=\varepsilon_L({h_1}^0\rightharpoonup 1_L){h_1}^1(h_2\rightharpoonup 1_L),$$
for all $h\in H$.
\end{lem}
\begin{proof}
Indeed, in the proof of Proposition \ref{barra embaixo coalgebra}, and thanks to Lemma \ref{igualdade importante},  we saw, in particular, that
$$x_1(h_1\rightharpoonup 1_L)_1\# {h_2}^0 \otimes x_2(h_1 \rightharpoonup 1_L)_2 {h_2}^1\# h_3=x_1 \underline{\#} {h_1}^0 \otimes x_2{h_1}^1 \underline{\#} h_2,$$
for all $x\in L$ and $h\in H$.

But, on the other side, we also have
$$x_1 \underline{\#} {h_1}^0 \otimes x_2 {h_1}^1\underline{\#} h_2  = x_1 (({h_1}^0)_1 \rightharpoonup 1_L) \# ({h_1}^0)_2 \otimes x_2 {h_1}^1 (h_2 \rightharpoonup 1_L) \# h_3.$$
Hence,
\begin{align*}
& x_1 (h_1 \rightharpoonup 1_L)_1\# {h_2}^0\otimes x_2(h_1 \rightharpoonup 1_L)_2{h_2}^1\# h_3 \\
= \, \, & x_1 ({(h_1}^0)_1 \rightharpoonup 1_L) \# {(h_1}^0)_2 \otimes x_2 {h_1}^1 (h_2 \rightharpoonup 1_L) \# h_3.
\end{align*}

And applying $\varepsilon_{L \# H} \otimes id_{L \# H}$
on this resultant equation and considering $x=1_L$, we get
$$\varepsilon_H({h_2}^0)(h_1\rightharpoonup 1_L){h_2}^1\# h_3 = \varepsilon_L({h_1}^0 \rightharpoonup 1_L){h_1}^1(h_2 \rightharpoonup 1_L)\# h_3.$$
Finally, applying  $id_L \otimes \varepsilon_H$ on  this last  equation (remember that $L \# H = L \otimes H$ as vector space) and using the equation from Definition \ref{defabeliano} we obtain
$$\varepsilon_H({h_1}^0){h_1}^1(h_2\rightharpoonup 1_L)=\varepsilon_L({h_1}^0\rightharpoonup 1_L){h_1}^1(h_2\rightharpoonup 1_L).$$
\end{proof}

Thanks to Lemma \ref{importante} we can rewrite the elements of $\overline{L\underline{\#}H}$ as follows
\begin{equation}
\overline{x\underline{\#}h}=\varepsilon_H({h_1}^0)x{h_1}^1\underline{\#}h_2.\label{novaescrita}
\end{equation}

\begin{thm}\label{principal}
Let $(H,L)$ be a quasi-abelian partial matched pair. Then $\overline{L\underline{\#} H}$ is a bialgebra with the following structure:
\begin{enumerate}
\item[(i)] Multiplication: $(\overline{x\underline{\#} h})(\overline{y \underline{\#} g})= \overline{x(h_1\rightharpoonup y)\underline{\#} h_2 g}$;
\item[(ii)] Unit: $\overline{1_L \underline{\#} 1_H}$;
\item[(iii)] Comultiplication: $\Delta(\overline{x\underline{\#}h})=\overline{x_1 \underline{\#} {h_1}^0} \otimes \overline{x_2 {h_1}^1 \underline{\#} h_2}$;
\item[(iv)] Counit: $\varepsilon(\overline{x\underline{\#} h})= \varepsilon_L(x)\varepsilon_L(h\rightharpoonup 1_L)$,
\end{enumerate}
for all $x, y\in L$ and $h,g\in H$.
\end{thm}
\begin{proof} Let $h,g\in H$ and $x,y\in L$. Then
\begin{eqnarray*}
&&\hspace{-0.65cm} (\overline{x \underline{\#} h})(\overline{y \underline{\#} g})\\
&\stackrel{(\ref{novaescrita})}{=}& (\varepsilon_L({h_1}^0\rightharpoonup 1_L)x{h_1}^1(h_2\rightharpoonup 1_L)\# h_3)(\varepsilon_H({g_1}^0)y{g_1}^1(g_2\rightharpoonup 1_L)\# g_3)\\
&=& \varepsilon_L({h_1}^0\rightharpoonup 1_L)\varepsilon_H({g_1}^0)x{h_1}^1(h_2\rightharpoonup 1_L)(h_3\rightharpoonup y{g_1}^1)(h_4\rightharpoonup (g_2\rightharpoonup 1_L))\# h_5 g_3\\
&=& \varepsilon_L({h_1}^0\rightharpoonup 1_L)\varepsilon_H({g_1}^0)x{h_1}^1(h_2\rightharpoonup y{g_1}^1)(h_3\rightharpoonup 1_L)(h_4g_2\rightharpoonup 1_L)\# h_5 g_3\\
&=& \varepsilon_L({h_1}^0\rightharpoonup 1_L)\varepsilon_H({g_1}^0)x{h_1}^1(h_2\rightharpoonup y{g_1}^1)(h_3 g_2\rightharpoonup 1_L)\# h_4 g_3.
\end{eqnarray*}

On the other hand,
\begin{eqnarray*}
&& \hspace{-0.7cm}\overline{x(h_1\rightharpoonup y)\underline{\#} h_2 g}{=}\\
&\stackrel{(\ref{novaescrita})}{=}&\hspace{-0.3cm} \varepsilon_H((h_2 g_1)^0)x(h_1\rightharpoonup y)(h_2 g_1)^1 (h_3g_2 \rightharpoonup 1_L)\# h_4g_3\\
&=&\hspace{-0.3cm} \varepsilon_H((h_2 g_1)^0)x\varepsilon_L((h_1\rightharpoonup y)_1)(h_1\rightharpoonup y)_2(h_2 g_1)^1(h_3 g_2\rightharpoonup 1_L)\# h_4g_3\\
&\stackrel{\ref{troca}}{=}& \hspace{-0.3cm} \varepsilon_H({h_3}^0{g_1}^0)x\varepsilon_L({h_1}^0\rightharpoonup y_1){h_1}^1(h_2\rightharpoonup y_2){h_3}^1(h_4\rightharpoonup {g_1}^1)(h_5 g_2\rightharpoonup 1_L)\# h_6 g_3\\
&\stackrel{3.5}{=}&\hspace{-0.3cm} \varepsilon_H({h_2}^0)\varepsilon_H({g_1}^0)x\varepsilon_L({h_1}^0\rightharpoonup y_1){h_1}^1{h_2}^1(h_3\rightharpoonup y_2)(h_4\rightharpoonup {g_1}^1)(h_5g_2\rightharpoonup 1_L)\# h_6 g_3\\
&\stackrel{\ref{partial matched pair}}{=}&\hspace{-0.3cm} \varepsilon_L({{h_1}^1}_2)\varepsilon_H({g_1}^0)x\varepsilon_L({{h_1}^0}_1\rightharpoonup 1_L)\varepsilon_L(y_1){h_1}^1(h_2\rightharpoonup y_2{g_1}^1)(h_3 g_2\rightharpoonup 1_L) \# h_4 g_3\\
&=&\hspace{-0.3cm} \varepsilon_H({g_1}^0)\varepsilon_L({h_1}^0\rightharpoonup 1_L)x{h_1}^1(h_2\rightharpoonup y{g_1}^1)(h_3 g_2\rightharpoonup 1_L)\# h_4 g_3.
\end{eqnarray*}
Thus, by Lemma \ref{algebra_smash} $\overline{L\underline{\#} H}$ is a unital algebra and by Proposition \ref{coalgebra cosmash2} also a counital coalgebra.

By Proposition \ref{barra embaixo coalgebra}, we already know that the maps $\Delta$ and $\varepsilon$ are both multiplicative and $\varepsilon(\overline{1_L \underline{\#} 1_H})=1_{\Bbbk}.$ Then, it remains  to check that $\Delta(\overline{1_L \underline{\#} 1_H})=\overline{1_L\underline{\#} 1_H} \otimes \overline{1_L\underline{\#} 1_H}$. Indeed,
\begin{eqnarray*}
\Delta(\overline{1_L \underline{\#} 1_H})&=& \varepsilon_L({1_H}^0\rightharpoonup 1_L) {{1_H}^1}_1\# {1_H}^0 \otimes {{1_H}^1}_2 {1_H}^1\# 1_H\\
&\stackrel{\ref{partial matched pair}}{=}& \varepsilon_L({1_H}^0 \rightharpoonup 1_L) {{1_H}^1}_1 \# \varepsilon_H({1_H}^0) 1_H\otimes {{1_H}^1}_2 {1_H}^1 \# 1_H\\
&\stackrel{\ref{part_comod_coalg}}{=}& \varepsilon_L({1_H}^{00}\rightharpoonup 1_L) {1_H}^{01}\# 1_H \otimes {1_H}^1\# 1_H\\
&\stackrel{\ref{partial matched pair}}{=}& \varepsilon_L({1_H}^0\rightharpoonup 1_L) \varepsilon_H({1_H}^0) {1_H}^1 \# 1_H \otimes {1_H}^1 \# 1_H\\
&=& \overline{1_L \underline{\#} 1_H} \otimes \overline{1_L \underline{\#} 1_H}.\\
\end{eqnarray*}

The proof is complete.
\end{proof}

\begin{rem} Take now a quasi-abelian partial matched pair $(H,L)$ and consider the subspace $L\overline{\#} H=(\varepsilon\otimes id_{L\otimes H})\Delta (L\# H),$ whose elements are as follows $x\overline{\#} h = \varepsilon_H({h_1}^0)x{h_1}^1\# h_2$. Then  by Proposition \ref{coalgebra cosmash2} the subspace $L\overline{\#} H$ is a coalgebra with counit and it is easy to check that it is also an algebra with left unit.
	
	In this way,  the subspace $\underline{L\overline{\#}H}=(L\overline{\#}H)(1_L\overline{\#}1_H),$ is a bialgebra with product $\underline{(x\overline{\#}h)}\underline{(y\overline{\#} g)}=\underline{x(h_1\rightharpoonup y)\overline{\#}h_2g}$, coproduct $\Delta(\underline{x\overline{\#}h})=\underline{x_1\overline{\#} {h_1}^0} \otimes \underline{x_2{h_1}^1\overline{\#} h_2}$, unit $\underline{1_L\overline{\#}1_H}$ and counit $\varepsilon(\underline{x\overline{\#}h})= \varepsilon_L(h\rightharpoonup 1_L)\varepsilon_L(x)$.

\vu
	
	Furthermore,
	\begin{eqnarray*}
		\underline{x\overline{\#}h}&=& \varepsilon_H({h_2}^0)x(h_1\rightharpoonup 1_L){h_2}^1\# h_3\\
		&\stackrel{3.5}{=}& \varepsilon_H({h_1}^0) x {h_1}^1 (h_2\rightharpoonup 1_L) \# h_3\\
		&=& \overline{x \underline{\#} h},
	\end{eqnarray*}
	for all $x\in L$ and $h\in H$. Therefore, the bialgebras $\underline{L\overline{\#}H}$ and $\overline{L\underline{\#}H}$ are the same, which justify to denote it simply $L\underline{\overline{\#}}H.$ From now on we will use only this last notation for such a bialgebra.
\end{rem}

The next result will be presented in order to give us the necessary conditions for $L\underline{\overline{\#}} H$ to be a Hopf algebra.

\begin{lem}\label{condicoes}
	Let $(H,L)$ be a quasi-abelian partial matched pair. If
	\begin{equation} \label{I'}
	(h^0\rightharpoonup 1_L)\otimes S_L(h^1)=\varepsilon_H({1_H}^0)\varepsilon_L(h\rightharpoonup 1_L)1_L \otimes {1_H}^1
	\end{equation}
	and
	\begin{equation}\label{II'}
	\varepsilon_H(h^0)(S_H(g)\rightharpoonup h^1)= \varepsilon_L(g\rightharpoonup 1_L)\varepsilon_H(h)\varepsilon_H({1_H}^0){1_H}^1
	\end{equation}
	hold for all $h,g\in H$, then
	\begin{enumerate}
		\item[(i)] $\varepsilon_L(h\rightharpoonup 1_L)=\varepsilon_L(S_H(h)\rightharpoonup 1_L)$;
		\item[(ii)] $h\rightharpoonup 1_L=\varepsilon_L(h\rightharpoonup 1_L)1_L$.
	\end{enumerate}
\end{lem}
\begin{proof}
	(i) Note that the equation (\ref{II'}) implies that
$$\varepsilon_H({h_1}^0)(S_H(h_2)\rightharpoonup {h_1}^1)= \varepsilon_L(h\rightharpoonup 1_L) \varepsilon_H({1_H}^0){1_H}^1.$$
	
	Applying $\varepsilon_L$ to both the sides of the such an equality and using the item (iv) of Definition \ref{partial matched pair} we get $\varepsilon_L(S_H(h)\rightharpoonup 1_L)= \varepsilon_L(h\rightharpoonup 1_L).$

\vu
		
	(ii) Applying $id_L\otimes \varepsilon_L$ in (\ref{I'}) one has
	$$(h^0\rightharpoonup 1_L)\varepsilon_L(S_L(h^1))=\varepsilon_H({1_H}^0)\varepsilon_L(h\rightharpoonup 1_L)1_L\varepsilon_L({1_H}^1).$$
	Therefore, $h\rightharpoonup 1_L=\varepsilon_L(h\rightharpoonup 1_L)1_L$.
	\end{proof}

\begin{thm}\label{anti}
	Under the conditions of Lemma \ref{condicoes}, $L\underline{\overline{\#}}H$ is a Hopf algebra with antipode given by
	$$S(x\underline{\overline{\#}}h)=(1_L\underline{\overline{\#}} S_H(h^0))(S_L(h^1)S_L(x)\underline{\overline{\#}}1_H).$$
\end{thm}
\begin{proof} By Theorem \ref{principal} is enough to check that
\[id_{L\underline{\overline{\#}}H}\ast S(x\due h)=\varepsilon(x \due h) 1_L \due 1_H=S\ast id_{L\underline{\overline{\#}}H}(x\due h).\]
Indeed, let $x\in L$ and $h\in H$,
	\begin{eqnarray*}
		id_{L\underline{\overline{\#}}H}\ast S(x\due h)&=& (x_1\due {h_1}^0)S(x_2 {h_1}^1\due h_2)\\
		&=& (x_1\due {h_1}^0)(1_L \due S_H({h_2}^0))(S_L({h_1}^1{h_2}^1)S_L(x_2)\due 1_H)\\
		&\stackrel{\ref{part_comod_coalg}}{=}& (x_1 \due {h^0}_1)(1_L\due S_H({h^0}_2))(S_L(h^1)S_L(x_2)\due 1_H)\\
		&=& (x_1 ({h^0}_1\rightharpoonup 1_L)\due {h^0}_2 S_H({h^0}_3))(S_L(h^1)S_L(x_2)\due 1_H)\\
		&=& (x_1 (h^0\rightharpoonup 1_L)\due 1_H)(S_L(h^1)S_L(x_2)\due 1_H)\\
		&=& x_1(h^0\rightharpoonup 1_L)S_L(h^1)S_L(x_2)\due 1_H\\
		&\stackrel{(\ref{I'})}{=}& x_1\varepsilon_H({1_H}^0)\varepsilon_L(h\rightharpoonup 1_L) {1_H}^1 S_L(x_2)\due 1_H\\
		&\stackrel{3.5}{=}& \varepsilon_H({1_H}^0) \varepsilon_L(h\rightharpoonup 1_L) x_1 S_L(x_2){1_H}^1 \due 1_H \\
		&=& \varepsilon_L(x)\varepsilon_L(h\rightharpoonup 1_L)\varepsilon_H({1_H}^0){1_H}^1 \due 1_H\\
		&=& \varepsilon(x\due h) 1_L\due 1_H
	\end{eqnarray*}
	
	On the other hand,
	\begin{eqnarray*}
		S\ast id_{L\underline{\overline{\#}}H}(x\due h)&=& S(x_1\due {h_1}^0)(x_2 {h_1}^1\due h_2)\\
		&=& (1_L \due S_H({h_1}^{00}))(S_L({h_1}^{01})S_L(x_1)\due 1_H)(x_2 {h_1}^1 \due h_2)\\
		&=& (1_L \due S_H({h_1}^{00}))(S_L({h_1}^{01})S_L(x_1)x_2{h_1}^1\due h_2)\\
		&\stackrel{\ref{part_comod_coalg}}{=}& \varepsilon_L(x) (1_L \due S_H({h_1}^0))(S_L({{h_1}^1}_1) {{h_1}^1}_2 \varepsilon_H({h_2}^0){h_2}^1\due h_3)\\
		&=& \varepsilon_L(x) (1_L \due S_H(h_1)) (\varepsilon_H({h_2}^0) {h_2}^1 \due h_3)\\
		&=& \varepsilon_L(x) (1_L \due S_H(h_1)) \varepsilon_H({h_2}^0)\varepsilon_H({h_3}^0) {h_2}^1 {h_3}^1 \underline{\#} h_4\\
		&\stackrel{\ref{part_comod_coalg}}{=}& \varepsilon_L(x)(1_L\due S_H(h_1))\varepsilon_H({h_2}^0) {h_2}^1\underline{\#} h_3\\
		&=& \varepsilon_L(x) (1_L\due S_H(h_1))(1_L \due h_2)\\
		&=& \varepsilon_L(x) (S_H(h_2)\rightharpoonup 1_L)\due S_H(h_1)h_3\\
		&\stackrel{\ref{condicoes}}{=}& \varepsilon_L(x) \varepsilon_L(S_H(h_2)\rightharpoonup 1_L) 1_L\due S_H(h_1)h_3 \\	
		&\stackrel{\ref{condicoes}}{=}& \varepsilon_L(x) \varepsilon_L(h_2\rightharpoonup 1_L) 1_L\due S_H(h_1)h_3  \\
		&=& \varepsilon_L(x) \varepsilon_L((h_2\rightharpoonup 1_L){h_3}^1) 1_L\due S_H(h_1){h_3}^0 \\
		&\stackrel{3.5}{=}& \varepsilon_L(x) \varepsilon_L({h_2}^1(h_3\rightharpoonup 1_L)) 1_L\due S_H(h_1){h_2}^0 \\
		&=& \varepsilon_L(x)\varepsilon_L(h\rightharpoonup 1_L) 1_L\due 1_H\\
		&=& \varepsilon(x \due h) 1_L \due 1_H.
	\end{eqnarray*}
\end{proof}
The Hopf  algebra  $L\due H$ will be called the \emph{partial bismash product} corresponding to the quasi-abelian partial matched pair $(H,L)$.

\begin{pro}\label{exemplo2} Under the conditions of Lemma \ref{exemplo}, if $(H,L)$ is a  quasi-abelian partial matched pair, then $L\overline{\underline{\#}}H$ is a Hopf algebra.
\end{pro}
\begin{proof} By Lemma \ref{exemplo} and Theorem \ref{principal} $L\overline{\underline{\#}}H$ is a bialgebra. Now, to show (\ref{I'}) and (\ref{II'}), it is enough to prove that $S_L(z)=z$ and $\lambda(S_H(h))=\lambda(h)$, for all $h\in H$.
	Indeed,
	\begin{eqnarray*}
		z &=& S_L(z_1)z_2z=S_L(z)z=z S_L(z)\\
		&=& z_1 S_L(z_2 z)= z_1 S_L(zz_2)=z_1 S_L(z_2) S_L(z)\\
		&=& S_L(z).
	\end{eqnarray*}
	And,
	\begin{eqnarray*}
		\lambda(h) &=& \lambda(h_1)\lambda(h_2 S_H(h_3))=\lambda(h_1)\lambda(S_H(h_2))\\
		&=& \lambda(S_H(h_1))\lambda(h_2)= \lambda(S_H(h_2))\lambda(S_H(h_1)h_3)\\
		&=& \lambda(S_H(h_1))\lambda(S_H(h_2)h_3)=\lambda(S_H(h)).
	\end{eqnarray*}
	Hence, by Theorem \ref{anti}, $L\overline{\underline{\#}}H$ is a Hopf algebra.
\end{proof}

\section{Examples and Properties}

 Now we present some examples and properties of Hopf algebras constructed using the theory developed in the previous section.

\begin{exa}\label{agora}
Consider $N$ a  subgroup of a finite group $G$ and $N'$ a subgroup of any group $G'$, then $\Bbbk G^{\ast}$ is a left partial $\Bbbk G'$-module algebra with partial action given by
	\begin{eqnarray*}
		\rightharpoonup: \Bbbk G'\otimes \Bbbk G^{\ast} & \longrightarrow & \Bbbk G^{\ast}\\
		g\otimes p_h & \longmapsto & \left\{
		\begin{array}{rl}
			1_{\Bbbk}p_h & \text{, } g\in N',\\
			0 & \text{, otherwise }
		\end{array} \right.
	\end{eqnarray*}
 and $\Bbbk G'$  is a right partial $\Bbbk G^{\ast}$-comodule coalgebra via $\rho(g)=g\otimes z$, for all $g\in G'$, where $z= \sum\limits_{h\in N} p_h\in\Bbbk G^{\ast}$. Then by Lemma \ref{exemplo} and Proposition \ref{exemplo2}, $\Bbbk G^{\ast}  \underline{\overline{\#}} \Bbbk G'$ is a Hopf algebra.
\end{exa}

\begin{exa} \label{exenormal}Let $N$ be a normal subgroup of a finite group $G$ such that $char(\Bbbk)\nmid|N|$. Therefore $\Bbbk G$ is a left partial $\Bbbk G^*$-module algebra given by
	\begin{eqnarray*}
		\rightharpoonup: \Bbbk G^*\otimes \Bbbk G & \longrightarrow & \Bbbk G\\
		p_g\otimes h & \longmapsto & \left\{
		\begin{array}{rl}
			\frac{1}{|N|}h & \text{, } g\in N,\\
			0 & \text{, otherwise }
		\end{array} \right.
	\end{eqnarray*}
and $\Bbbk G^*$ is a right partial $\Bbbk G$-comodule coalgebra given by $\rho(p_g)=p_g \otimes z$, for all $g\in G$, with $z = \sum\limits_{g\in N} \frac{1}{|N|} g$. Again by Lemma \ref{exemplo} and Proposition \ref{exemplo2}, $\Bbbk G \due \Bbbk G^*$ is a Hopf algebra.
\end{exa}

\begin{exa}
	Let $H$ be a Hopf algebra with invertible antipode. Consider $H$ a left partial $H^{op}$-module algebra given by $h\rightharpoonup a=S(h_1)ah_3\lambda(h_2)$, such that $\lambda: H^{op}\longrightarrow \Bbbk$ satisfies \[\lambda(h)\lambda(l)=\lambda(h_1)\lambda(h_2 \cdot_{op} l), \lambda(1_H)= 1_{\Bbbk}\,\ \text{and}\,\ \lambda(h_1)h_2 = h_1\lambda(h_2).\] Also consider $H^{op}$ a right partial $H$-comodule coalgebra with partial coaction given by $\rho(h)=h_2\otimes S(h_1)zh_3$, where z is central element in $H$ such that $\varepsilon(z) = 1_{\Bbbk}$ and $z \otimes z = \Delta(z)(1 \otimes z)$
	
	Then $(H^{op}, H)$ is a quasi-abelian matched pair that satisfies (\ref{I'}) and (\ref{II'}), hence $H \underline{\overline{\#}} H^{op}$ is a Hopf algebra.
\end{exa}

We are also able to construct Hopf algebras from partial (co)actions even if the partial matched pair is not necessarily quasi-abelian.

\begin{pro}\label{calma} Let $L$ be the left partial $H$-module algebra given by Example \ref{mod_acao_corpo} and let $H$ be the right partial $L$-comodule coalgebra from Example \ref{prop cc}.	Then $(H,L)$ is a partial matched pair if and only if the following properties hold for all $x\in L$ and $h\in H$:
	\begin{enumerate}
		\item[(i)] $xz = zxz$;
		\item[(ii)] $\lambda(h_1)h_2\lambda(h_3) = \lambda(h_1)h_2$.
	\end{enumerate}
\end{pro}
\begin{proof}
	It is analogous to the proof of Lemma \ref{exemplo}.
\end{proof}

\begin{cor}\label{melhor} Under the same conditions of  Proposition \ref{calma}, $(H,L)$ is a partial matched pair if and only if $L \underline{\overline{\#}} H$ is a Hopf algebra.
\end{cor}
\begin{proof}
	Straightforwardly from the definition of $L \underline{\overline{\#}} H$.
\end{proof}

\begin{exa}
	Let $G$ be a group, $N$ a subgroup of $G$ and $H$ a normal subgroup of $G$ such that $char(\Bbbk)\nmid|H|$. Consider $\Bbbk G$ a left partial $\Bbbk G$-module algebra given by $g \rightharpoonup h =\lambda(g)h$, with $\lambda(g)=1_{\Bbbk}$ if $g\in N$, and $\lambda(g)=0$ otherwise. Also, regard $\Bbbk G$ as a right partial $\Bbbk G$-comodule coalgebra with partial coaction given by $\rho(g)=g \otimes z$, where $z=\sum\limits_{n\in H} \frac{1}{|H|}n$. Then, $(\Bbbk G, \Bbbk G)$ is a partial matched pair and, by Corollary \ref{melhor}, $\Bbbk G \underline{\overline{\#}} \Bbbk G$ is a Hopf algebra.
\end{exa}

\begin{rem}\label{h4} Let $L$ be a Hopf algebra and $\mathbb{H}_4=\Bbbk\{1,c,x,xc\}$ the Sweedler algebra (in this case $char(\Bbbk)\neq 2$ is assumed). Given $y\in L$ and $a\in \mathbb{H}_4$, consider the partial action of $\mathbb{H}_4$ on $L$ given by $a\rightharpoonup y=\lambda(a)y$, where $\lambda(1)=1_{\Bbbk}$, $\lambda(c)=0$ and $\lambda(xc)=-\lambda(x)$.
Under such assumptions, $L \underline{\overline{\#}} \mathbb{H}_4$ is not a Hopf algebra for any partial coaction of $L$ on $\mathbb{H}_4$, because $\lambda$ does not satisfy the item (ii) of Proposition \ref{calma}.  

\vu

Likewise, if we consider the partial coaction of $L$ on $\mathbb{H}_4$ given by $\rho(y)=y\otimes z$,	for all $y\in L$ and $z=\frac{1}{2} + \frac{1}{2}c + \beta xc$, with $\beta\in\Bbbk$, then $\mathbb{H}_4 \underline{\overline{\#}} L$ is also not a Hopf algebra, for any partial action of $L$ on $\mathbb{H}_4$, because the item (i) of Proposition \ref{calma} is not satisfied.
\end{rem}

\begin{exa}
	Let $A_{2,2}=\Bbbk\langle g,h,x \ | \ g^2=h^2=1, x^2=0, gx=-xg, hx=-xh, gh=hg\rangle$ be the Hopf algebra with coproduct given by $$\Delta(g)=g\otimes g, \ \Delta(h)=h\otimes h \ \mbox{and} \ \Delta(x)=x\otimes 1 + g\otimes x.$$ For $a,b\in A_{2,2}$, consider the partial action of $A_{2,2}$ on $A_{2,2}$ such that $a\rightharpoonup b=\lambda(a)b$, where  $$\lambda(1)=\lambda(g)=1, \quad\lambda(h) =\lambda(gh) =\lambda(x)= \lambda(gx) = \lambda(hx) =\lambda(ghx)= 0.$$ Then $\lambda(a_1)a_2=a_1\lambda(a_2)$, for all $a\in A_{2,2}$.

\vu
		
		Also, consider the partial coaction of $A_{2,2}$ on $A_{2,2}$ such that $\rho(b)=b\otimes z$,	for all $b\in A_{2,2}$, where $z=\frac{1+gh}{2}$ is a central element in $ A_{2,2}$ (also here, $char(\Bbbk)\neq 2$ ).
		
		In this case, also thanks to Lemma \ref{exemplo} and Proposition \ref{exemplo2}, $A_{2,2}\due A_{2,2}$ is a Hopf algebra. Note that, in this case, $A_{2,2}\due A_{2,2}\cong \mathbb{H}_4\otimes \mathbb{H}_4$ as Hopf algebras.
\end{exa}

\begin{exa}\label{exemplosemisimples}
	Assume again that $char(\Bbbk)\neq 2$ and let $A_4'$ and $H$ be the following Hopf algebras:
\begin{itemize}	
\item	$H := \Bbbk   \langle  g,\ h,\ x, \ | \ g^4=h^2=(gh)^4=1, \ x^2=0, \ xg=qgx, \ xh=hx \rangle,$ where $q$ is a fourth-primitive root of $1$, with coproduct given by
		$$\Delta (x) = x \otimes 1 + g^2h \otimes x, \ \Delta (g) = g \otimes g \  \textrm{ and } \ \Delta (h) = h \otimes h; $$
		
\item		$A_4':= \Bbbk\langle c,y \ | \ c^4=1, \ y^2=0, \ cy=-yc\rangle$ with coproduct given by $\Delta(c)=c\otimes c$ and $\Delta(y)=y \otimes 1+ c\otimes y$.
\end{itemize}	
	
Consider the partial action of $H$ on $A_4'$ such that $a\rightharpoonup b=\lambda(a)b$, for all $a\in H$ and $b\in A_4'$, where
			$$\lambda(1) = \lambda(h)= \lambda(g^2)= \lambda(g^2h) = 1, \;\;\;  \lambda(g)= \lambda(g^3) = \lambda(gh) = \lambda(g^3h) = 0,$$
		$$\lambda(x)= \lambda(gx)= \lambda(g^2x)= \lambda(g^3x)= \lambda(hx)= \lambda(ghx)=  \lambda(g^2hx) = \lambda(g^3hx) = 0.$$
		Then $\lambda(a_1)a_2=a_1\lambda(a_2)$, for all $a\in H$.

\vu
		
		Also, consider the partial coaction of $A_4'$ on $H$ such that $\rho(b)=b\otimes z$,	for all $b\in H$, where $z=\frac{1+c^2}{2}$ is a central element in $ A_4'$.

\vu
		
		Thus, by Lemma \ref{exemplo} and Proposition \ref{exemplo2}, $A_4'\due H$ is a Hopf algebra. In this case, $A_4'\due H \cong \mathbb{H}_4\otimes A_{2,2} $ as Hopf algebras.
		\end{exa}

\begin{pro}\label{prointegral}
Let $\Bbbk$ be an algebraically closed field of characteristic zero, $L$ and $H$ finite dimensional Hopf algebras and $(H,L)$ an abelian partial matched pair that satisfies the conditions  \emph{(\ref{I'})} and \emph{(\ref{II'})}. If $0\neq\alpha\in \int_l^L$ and $0\neq t\in \int_l^H$, then $\alpha\underline{\overline{\#}} t\in\int_l^{L\underline{\overline{\#}} H}$.
\end{pro}
\begin{proof}
	First, we show that $h\rightharpoonup \alpha \in \int_{l}^{L}$, for all $h\in H$. Indeed, let $x\in L$,
	\begin{eqnarray*}
		x(h\rightharpoonup \alpha) &=& (h_2S_H(h_3)\rightharpoonup x)(h_1\rightharpoonup \alpha)\\
		&\stackrel{\ref{defabeliano}}{=}& (h_1\rightharpoonup \alpha)(h_2S_H(h_3)\rightharpoonup x)\\
		&=& h_1 \rightharpoonup (\alpha(S_H(h_2)\rightharpoonup x))\\
		&\stackrel{\ref{defabeliano}}{=}& h_1\rightharpoonup ((S_H(h_2)\rightharpoonup x)\alpha)\\
		&=& \varepsilon_L(S_H(h_2)\rightharpoonup x)(h_1\rightharpoonup \alpha)\\
		&\stackrel{\ref{partial matched pair}}{=}& \varepsilon_L(S_H(h_2)\rightharpoonup 1_L)\varepsilon_L(x)(h_1\rightharpoonup \alpha)\\
		&\stackrel{\ref{condicoes}}{=}& \varepsilon_L(h_2\rightharpoonup 1_L)\varepsilon_L(x)(h_1\rightharpoonup \alpha)(h_2\rightharpoonup 1_L)\\
		&\stackrel{\ref{igualdade importante}}{=}& \varepsilon_L(x)(h\rightharpoonup \alpha).
	\end{eqnarray*}

Thus, $h \rightharpoonup \alpha = k \alpha$, for some $k\in \Bbbk$. Also, since $L$ is commutative we have that $L$ is semisimple, that is, $\varepsilon_L(\alpha) \neq 0$. So,
\begin{eqnarray*}
k&=&\varepsilon_L(k\alpha)(\varepsilon_L(\alpha))^{-1}\\
&=& \varepsilon_L(h \rightharpoonup \alpha)(\varepsilon_L(\alpha))^{-1}\\
&\stackrel{\ref{partial matched pair}}{=}& \varepsilon_L(h \rightharpoonup 1_L)\varepsilon_L(\alpha)(\varepsilon_L(\alpha))^{-1}\\
&=& \varepsilon_L(h\rightharpoonup 1_L),
\end{eqnarray*}
for all $h\in H$. Then
	\begin{eqnarray*}
		(x\underline{\overline{\#}} h)(\alpha\underline{\overline{\#}} t)&=& x(h_1\rightharpoonup \alpha)\underline{\overline{\#}} h_2t\\
		&=& x(h\rightharpoonup \alpha) \underline{\overline{\#}} t\\
		&=& \varepsilon_L(x)\varepsilon_L(h\rightharpoonup 1_L)\alpha\underline{\overline{\#}}t\\
		&=& \varepsilon(x\underline{\overline{\#}} h)\alpha\underline{\overline{\#}}t,
	\end{eqnarray*}
for all $(x\underline{\overline{\#}} h) \in (L\underline{\overline{\#}} H)$. Therefore, $\alpha\underline{\overline{\#}} t\in\int_l^{L\underline{\overline{\#}} H}$.
\end{proof}

\begin{cor}\label{smash semis} Let $(H,L)$ be the partial matched pair of Proposition \ref{calma}, with $L$ and $H$ both finite dimensional Hopf algebras. Then $L \underline{\overline{\#}} H$ is a Hopf algebra and the following statements are equivalent:
	\begin{enumerate}
		\item[(i)] $L \underline{\overline{\#}} H$ is semisimple;
		\item[(ii)] $L \underline{\overline{\#}} H$ is separable;
		\item[(iii)] $L$ is semisimple and $\lambda\left(\int_{l}^{H}\right)\neq 0$ or $\lambda\left(\int_{r}^{H}\right)\neq 0$.
	\end{enumerate}
\end{cor}
\begin{proof} The first assertion on $L \underline{\overline{\#}} H$ is ensured by Corollary \ref{melhor}.

\vu

	(i)$\Longleftrightarrow$ (ii) It is a trivial and well known result from the classical theory of Hopf algebras.
	
	(i)$\Longleftrightarrow$ (iii) Let $\alpha \underline{\overline{\#}} t\in \int_{l}^{L\underline{\overline{\#}} H}$ with $\alpha \in \int_{l}^{L}$ and $t \in \int_{l}^{H}$. Then
	\begin{eqnarray*}
		L \underline{\overline{\#}} H  \mbox{\ is \ semisimple} &\Leftrightarrow &\varepsilon(\alpha \underline{\overline{\#}} t)\neq 0\\
		&\Leftrightarrow&\varepsilon_L(\alpha)\neq 0 \ \mbox{and} \ \lambda(t)\neq 0 \\
		&\Leftrightarrow& L \mbox{ \ is \ semisimple \ and \ } \lambda\left(\int_{l}^{H}\right)\neq 0.
	\end{eqnarray*}
\end{proof}

\begin{pro}\label{dual do par combinado}
	Let $H$ and $L$ be finite dimensional Hopf algebras. If $(H,L)$ is a quasi-abelian partial matched pair, then $(L^{\ast}, H^{\ast})$ is a quasi-abelian partial matched pair, where $H^*$ is a right partial $L^*$-module algebra and $L^*$ is a left partial $H^*$-comodule coalgebra.
\end{pro}
\begin{proof} Let $(H,L)$ be a quasi-abelian partial matched pair. Since $H$ is a right partial $L$-comodule coalgebra via $\rho(h)=h^0\otimes h^1$, for all $h\in H$, it follows that $H^{\ast}$ is a right partial $L^{\ast}$-module algebra via
	$$(\phi\leftharpoonup f)(h)=\phi(h^0)f(h^1),$$
		for all $f\in L^{\ast}$, $\phi\in H^{\ast}$ and $h\in H$. Similarly, since $L$ is a left partial $H$-module algebra via $\rightharpoonup$, then $L^{\ast}$ is a  left partial $H^{\ast}$-comodule coalgebra via
		$$\rho(f)= \sum\limits_{i=1}^{n} {h_i}^{\ast} \otimes f\leftharpoondown h_i,$$
		where $(f\leftharpoondown h_i)(x)=f(h_i\rightharpoonup x)$, for all $f\in L^{\ast}$, $x\in L$ and $\{h_i\}_{i=1}^{n}$ a basis of $H$ with dual basis $\{h_i^*\}_{i=1}^{n}$.
		
It is straightforward to check that $(L^{\ast}, H^{\ast})$ is a quasi-abelian partial matched pair.
\end{proof}

As a consequence, we have the following result.

\begin{cor}\label{cor dual hopf}
	Under the previous conditions, if $(H,L)$ satisfies \emph{(\ref{I'})} and \emph{(\ref{II'})}, then $H^{\ast}\overline{\underline{\#}}L^{\ast}$ is a Hopf algebra.
\end{cor}
\begin{proof}
	It is enough to apply Theorem \ref{anti} to the quasi-abelian partial matched pair $(L^{\ast},H^{\ast})$.
\end{proof}

\begin{thm}\label{tobem}
	Let $H$ and $L$ be finite dimensional Hopf algebras. If $(H,L)$ is a quasi-abelian partial matched pair for which the equations \emph{(\ref{I'})} and \emph{(\ref{II'})} hold, then
		$$(L \underline{\overline{\#}} H)^{\ast} \simeq H^{\ast} \underline{\overline{\#}} L^{\ast}$$
	as Hopf algebras.
\end{thm}
\begin{proof}
	Consider the liner map
	\begin{eqnarray*}
		\theta: H^{\ast}\# L^{\ast} & \longrightarrow & (L\underline{\overline{\#}} H)^{\ast}\\
		\varphi\# f & \longmapsto & \theta(\varphi \# f)( x\underline{\overline{\#}} h)=\varepsilon_H({h_1}^0) \varphi(h_3)f(x{h_1}^1(h_2\rightharpoonup 1_L)).
	\end{eqnarray*}	
	
	Note that $\theta(\varphi\underline{\overline{\#}} f)(x\underline{\overline{\#}} h)=\theta(\varphi\# f)(x\underline{\overline{\#}} h)$, for all $\varphi\underline{\overline{\#}} f\in H^{\ast}\underline{\overline{\#}} L^{\ast}$ and $(x\underline{\overline{\#}} h)\in L\underline{\overline{\#}} H$. Thus, the restriction
	\begin{eqnarray*}
		\theta: H^{\ast}\underline{\overline{\#}} L^{\ast} & \longrightarrow & (L\underline{\overline{\#}} H)^{\ast}
		\end{eqnarray*}	
 is well defined and its inverse $\theta^{-1}: (L\underline{\overline{\#}} H)^{\ast}  \longrightarrow H^{\ast}\underline{\overline{\#}} L^{\ast}$ is given by $\theta^{-1}(\xi)=\sum\limits_{i=1}^{n} {h_i}^{\ast} \underline{\overline{\#}} \xi_i$, for all $\xi\in (L\underline{\overline{\#}} H)^{\ast}$, where $\{h_i\}_{i=1}^{n}$ is a basis of $H$ with dual basis $\{h_i^*\}_{i=1}^{n}$ and $\xi_i\in L^{\ast}$ is such that $\xi_i(x)= \xi(x\underline{\overline{\#}} h_i)$, for all $x\in L$.
	
	Moreover, $\theta$ is a morphism of algebras. Indeed,
		\begin{eqnarray*}
	& &\hspace{-0.7cm}(\theta( \varphi \underline{\overline{\#}} f)\theta(\phi \underline{\overline{\#}} g))(x\underline{\overline{\#}} h)=\\
	&=& \hspace{-0.3cm}\theta( \varphi \underline{\overline{\#}} f)(x_1\underline{\overline{\#}} {h_1}^0)\theta(\phi \underline{\overline{\#}} g)(x_2{h_1}^1\underline{\overline{\#}} h_2)\\
	&=& \hspace{-0.3cm}\varepsilon_H({{{h_1}^0}_1}^0) \varphi({{h_1}^0}_3) f(x_1 {{{h_1}^0}_1}^1({{h_1}^0}_2\rightharpoonup 1_L)) \varepsilon_H({h_2}^0) \phi(h_4)g(x_2 {h_1}^1 {h_2}^1 (h_3\rightharpoonup 1_L))\\
	&=& \hspace{-0.3cm}\varepsilon_H({{{h_1}^0}_1}^0) \varphi({{h_1}^0}_3)\phi(h_3)f(x_1 {{{h_1}^0}_1}^1 ({{h_1}^0}_2\rightharpoonup 1_L))g(x_2 {{h_1}^1} (h_2\rightharpoonup 1_L))\\
	&=& \hspace{-0.3cm}\varepsilon_H({{{h_1}^0}_1}^0) \varphi({h_2}^0)\phi(h_4)f(x_1 {{{h_1}^0}_1}^1 ({{h_1}^0}_2\rightharpoonup 1_L))g(x_2 {{h_1}^1} {h_2}^1(h_3\rightharpoonup 1_L))\\
	&=& \hspace{-0.3cm}\varepsilon_H({h_1}^{00}) \varphi({h_3}^0)\phi(h_5)f(x_1 {h_1}^{01} ({h_2}^0\rightharpoonup 1_L))g(x_2 {{h_1}^1} {h_2}^1{h_3}^1(h_4\rightharpoonup 1_L))\\
	&=& \hspace{-0.3cm}\varepsilon_H({h_1}^0) \varphi({h_3}^0)\phi(h_5)f(x_1 {{h_1}^1}_1 ({h_2}^0\rightharpoonup 1_L))g(x_2 {{h_1}^1}_2 {h_2}^1 {h_3}^1 (h_4\rightharpoonup 1_L))\\
	 &\stackrel{\ref{igualdade importante}}{=}& \hspace{-0.3cm}\varepsilon_H({h_1}^0) \varphi({h_3}^0)\phi(h_4) f(x_1 {{h_1}^1}_1 (h_2\rightharpoonup 1_L)_1) g(x_2 {{h_1}^1}_2 (h_2\rightharpoonup 1_L)_2 {h_3}^1)\\
	 &=& \hspace{-0.3cm}\varepsilon_H({h_1}^0) [(\varphi \leftharpoonup g_2)\phi](h_3) (fg_1)(x {h_1}^1 (h_2\rightharpoonup 1_L))\\
	 &=&\hspace{-0.3cm}\theta((\varphi \underline{\overline{\#}} f)(\phi \underline{\overline{\#}} g)) (x \underline{\overline{\#}} h).
	\end{eqnarray*}
	And, \[\theta(\varepsilon_H \underline{\overline{\#}} \varepsilon_L)( x \underline{\overline{\#}} h) = \varepsilon_H({h_1}^0) \varepsilon_H(h_3) \varepsilon_L( x {h_1}^1 (h_2 \rightharpoonup 1_L))= \varepsilon(x \underline{\overline{\#}} h).\]

	\vu

	It remains to verify that $\theta$ is a morphism of coalgebras. That is,
	\begin{eqnarray*}
	&\hspace{-0.7cm}&\hspace{-0.85cm}\Delta(\theta(\varphi \underline{\overline{\#}} f)) (x \underline{\overline{\#}} h \otimes y \underline{\overline{\#}} g)=\\
	&\hspace{-0.7cm}=& \hspace{-0.4cm}\theta( \varphi \due f) (x(h_1\rightharpoonup y) \due h_2g)\\
		&\hspace{-0.7cm}=& \hspace{-0.4cm}\varepsilon_H(({h_2g_1})^0) \varphi(h_4g_3)f(x(h_1\rightharpoonup y)(h_2g_1)^1(h_3g_2\rightharpoonup 1_L))\\
		&\stackrel{\hspace{-0.7cm}\ref{troca}}{\hspace{-0.7cm}=}& \hspace{-0.4cm}\varepsilon_H({h_3}^0{g_1}^0)\varphi(h_6g_3) f(x\varepsilon_L({h_1}^0\rightharpoonup y_1){h_1}^1 (h_2\rightharpoonup y_2){h_3}^1(h_4\rightharpoonup {g_1}^1)(h_5g_2\rightharpoonup 1_L))\\
		&\stackrel{\hspace{-0.7cm}\ref{partial matched pair}}{\hspace{-0.7cm}=}&\hspace{-0.4cm} \varepsilon_H({h_3}^0{g_1}^0)\varphi(h_5g_3)f(x\varepsilon_L({h_1}^0\rightharpoonup 1_L){h_1}^1 (h_2\rightharpoonup y){h_3}^1(h_4\rightharpoonup ({g_1}^1(g_2\rightharpoonup 1_L))))\\
		&\stackrel{\hspace{-0.7cm}3.5}{\hspace{-0.7cm}=}&\hspace{-0.4cm} \varepsilon_H({h_2}^0{g_1}^0)\varphi(h_5g_3)f(x\varepsilon_L({h_1}^0\rightharpoonup 1_L){h_1}^1{h_2}^1(h_3\rightharpoonup y)(h_4\rightharpoonup ({g_1}^1(g_2\rightharpoonup 1_L))))\\
		&\hspace{-0.7cm}=&\hspace{-0.4cm}\varepsilon_H({g_1}^0)\varphi(h_4g_3)f(x\varepsilon_L({h_1}^0\rightharpoonup 1_L){h_1}^1(h_2\rightharpoonup 1_L)(h_3\rightharpoonup (y{g_1}^1(g_2\rightharpoonup 1_L))))\\
		&\stackrel{\hspace{-0.7cm}\ref{importante}}{\hspace{-0.7cm}=}&\hspace{-0.4cm} \varepsilon_H({g_1}^0)\varphi(h_3g_3)f(x\varepsilon_L({h_1}^0){h_1}^1(h_2\rightharpoonup (y{g_1}^1(g_2\rightharpoonup 1_L))))\\
		&\hspace{-0.7cm}=& \hspace{-0.4cm}\varepsilon_H({h_1}^0) \varepsilon_H({g_1}^0)\varphi(h_4 g_3) f_1 (x {h_1}^1 (h_2\rightharpoonup 1_L))f_2(h_3\rightharpoonup (y{g_1}^1(h_2\rightharpoonup 1_L)))\\
		&\hspace{-0.7cm}=& \hspace{-0.4cm} \sum\varepsilon_H({h_1}^0){h_i}^*(h_3)\varphi(h_4g_3)f_1(x{h_1}^1(h_2\rightharpoonup 1_L))\varepsilon_H({g_1}^0)(f_2\leftharpoondown h_i)(y{g_1}^1(g_2\rightharpoonup 1_L))\\
		&\hspace{-0.7cm}=&\hspace{-0.4cm} \varepsilon_H({h_1}^0){f_2}^{-1}(h_3)\varphi_1(h_4)f_1(x{h_1}^1(h_2\rightharpoonup 1_L))\varepsilon_H({g_1}^0)\varphi_2(g_3){f_2}^0(y{g_1}^1(g_2\rightharpoonup 1_L))\\
		&\hspace{-0.7cm}=& \hspace{-0.4cm}\theta({f_2}^{-1}\varphi_1 \underline{\overline{\#}} f_1)(x \underline{\overline{\#}} h) \theta(\varphi_2 \underline{\overline{\#}} {f_2}^0)(y \underline{\overline{\#}} g)\\
		&\hspace{-0.7cm}=&\hspace{-0.4cm}(\theta \otimes \theta) \circ \Delta(\varphi \underline{\overline{\#}} f)(x \underline{\overline{\#}} h \otimes y \underline{\overline{\#}} g).
	\end{eqnarray*}
And,
	\begin{eqnarray*}
		\varepsilon(\varphi \due f) &=& \varepsilon(\varepsilon_{L^*}({f_3}^0)(1_{L^*} \leftharpoondown f_2){f_3}^{-1}\varphi \# f_1)\\
		&=& \varepsilon_{H^*}(\varepsilon_L\leftharpoondown f)\varphi(1_H)\\
		&=& f({1_H}^1)\varepsilon_L({1_H}^0)\varphi(1_H)\\
		&=& \theta(\varphi \due f)(1_L \due 1_H)\\
		&=& \varepsilon\circ\theta(\varphi \due f).
	\end{eqnarray*}
Therefore, $\theta$ is a isomorphism of Hopf algebras.
	\end{proof}

\end{document}